\documentclass[11pt]{amsart}
\usepackage{amsmath,amstext, amsthm,amssymb,comment}

\numberwithin{equation}{section}
\theoremstyle{plain}
\newtheorem{theorem}{Theorem}

\newtheorem{cor}{Corollary}
\newtheorem{exam}{Example}
\newtheorem{defn}{Definition}
\newtheorem{prop}{Proposition}

\title{Refinements of the Rogers-Ramanujan identities}
\author{Kathleen O'Hara}
\address{ Virginia Bioinformatics Institute, Virginia Polytechnic Institute
and State University, Blacksburg VA 24061}
\email{kohara1@vbi.vt.edu}
\thanks{ }
\author{Dennis Stanton}
\address{School of Mathematics, University of Minnesota, Minneapolis MN 55455}
\email{stanton@math.umn.edu}
\thanks{Research supported by NSF grant DMS 11-48634}

\begin{document}
\maketitle
\date{October 13, 2014}
\begin{abstract} Refinements of the classical Rogers-Ramanujan 
identities are given in which some parts are weighted. Combinatorial interpretations 
refining MacMahon's results are corollaries.
\end{abstract}

\section{Introduction}

The two Rogers-Ramanujan identities are
\begin{equation}
\label{RR}
\begin{aligned}
1+\sum_{n=1}^\infty \frac{q^{n^2}}{(1-q)(1-q^2)\cdots (1-q^n)}=&
\ \frac{1}{\prod_{k=0}^\infty (1-q^{5k+1})(1-q^{5k+4})},\\
\quad 
1+\sum_{n=1}^\infty \frac{q^{n^2+n}}{(1-q)(1-q^2)\cdots (1-q^n)}=&
\ \frac{1}{\prod_{k=0}^\infty (1-q^{5k+2})(1-q^{5k+3})}.
\end{aligned}
\notag
\end{equation}
These two identities, which in this paper are referred to as the first and the second 
Rogers-Ramanujan identity, have been extensively studied and generalized. 
See \cite{And2} for a selection of proofs  to these identities. 
Andrews \cite{And3} gave a generalization to all odd moduli, and 
Bressoud \cite{Br} to all moduli. 

MacMahon and Schur gave the combinatorial meaning of these identities. 
The restatement of the first identity is 
\begin{prop} The number of partitions of $N$ into parts whose consecutive differences are 
at least two is equal to the number of partitions of $N$ into parts which are congruent to
$1$ or $4$ modulo $5$. 
\end{prop}
The second identity has a similar interpretation \cite[Chapter 7.3]{And4}, 
using no $1$'s in the difference two partition and parts congruent to $2$ or $3$ modulo 5.  

The analytic proofs of the Rogers-Ramanujan identities establish a 
generalization which replaces both sides by sums, 
and then evaluates a specialization of one side as a product, often using the Jacobi product identity.
Two proofs which do not are \cite{And1} (using Watson's transformation) 
and \cite[Theorem 7.1]{GIS} 
(using a quintic transformation). No results are given in which the individual 
parts on the product side are weighted, for example weighting each part of size $13$  
in the second Rogers-Ramanujan identity by $t$. 
This would replace  the factor $1-q^{13}$ in the denominator product by $1-tq^{13}.$ 

\vfill\pagebreak

The purpose of this paper is to find the corresponding sum side if 
part sizes on product side are given weights. 
Surprisingly, all of our sum sides are manifestly non-negative.
Moreover we shall interpret these sum sides combinatorially, thereby refining MacMahon's result.
If one could give weights to all parts independently, one would 
have a direct Rogers-Ramanujan bijection. We do not this, but do give 
results for the following choices of weighted parts on the product side: 
\begin{enumerate}
\item an arbitrary single part (Theorems~\ref{partM} and  \eqref{partMeq}),
\item some families of two distinct parts (Theorems~\ref{twopartM} and \ref{twopart14}),
\item any subset of the four smallest part sizes (Theorems~\ref{twvx23theorem}, \ref{twvx14thm}).
\end{enumerate}
 The refinements of MacMahon's combinatorial interpretations are given in   
Theorems~\ref{generalminithm} and  \ref{generalmini14thm},
and Corollaries ~\ref{general2partcor}, \ref{general2part14cor}, 
\ref{firstbigcomb}, and \ref{bigcomb}.
Three curious specialized bijective results, Theorems~\ref{spec1}, \ref{spec2} and 
\ref{spec3}, are given in \S5. 

We put 
$$
[M+1]_q=1+q+q^2+\cdots+q^M
$$
for a non-negative integer $M$.

\section{Refinements with one part}

As a warm-up result, let us see what happens if each part of size $2$ in the second 
Rogers-Ramanujan identity is weighted by $t.$

\begin{prop} 
\label{miniprop}
A $t$-refinement of the second Rogers-Ramanujan identity is
$$
\begin{aligned}
1+q^2\frac{(t+q)}{1-tq^2}+
&\sum_{m=2}^\infty 
\frac{q^{m(m+1)}}{(1-q)(1-tq^2)(1-q^3)(1-q^4)\cdots (1-q^m)} \\
&=\frac{1}{(1-tq^2)(1-q^3)(1-q^7)(1-q^8)(1-q^{12})(1-q^{13})\cdots}.
\end{aligned}
$$
\end{prop}

\begin{proof} Multiply the second Rogers-Ramanujan identity by $\frac{1-q^2}{1-tq^2}$, 
and use
\begin{equation}
\label{weirdeq}
\frac{1-q^2}{1-tq^2} \left( 1+\frac{q^2}{1-q}\right)= 1+\frac{q^2(t+q)}{1-tq^2}.
\end{equation}
\end{proof}

For the combinatorial interpretations we need a definition. 
Recall that the columns of the Ferrers diagram of a partition $\lambda$ 
are the rows of the conjugate partition $\lambda^T$.

\begin{defn} Let ${\sf{Diff_2}}$ be the set of partitions whose consecutive 
differences are at least two. Let ${\sf{Diff_2}}^*$ be the subset of ${\sf{Diff_2}}$ 
consisting of 
partitions with no $1$ as a part.  
\begin{enumerate}
\item If $\lambda=(\lambda_1,\cdots,\lambda_m)\in {\sf{Diff_2}}, $ let
$col(\lambda)$ be the partition formed by the columns 
of $(\lambda_1-(2m-1),\cdots, \lambda_m -1).$ 
\item If $\lambda=(\lambda_1,\cdots,\lambda_m)\in {\sf{Diff_2}}^*,$ let $col^*(\lambda)$ 
be the partition formed by the columns of $(\lambda_1-2m,\cdots, \lambda_m -2).$ 
\end{enumerate}
\end{defn}

\begin{exam} If $\lambda=(16,12,7,4,1)\in {\sf{Diff_2}}$, 
$$col(\lambda)=(7,5,2,1,0)^T=(4,3,2,2,2,1,1).$$
If $\lambda=(13,10,6,4)\in {\sf{Diff_2}}^*$,
$$col^*(\lambda)=(5,4,2,2)^T=(4,4,2,2,1).$$
\end{exam}

We can give a weight to any part, not just $2,$ using the next proposition which generalizes
\eqref{weirdeq}

\begin{prop} For any positive integer $M$,
$$
\begin{aligned}
\frac{1-q^{M+1}}{1-tq^{M+1}}& \sum_{k=0}^M \frac{q^{k(k+1)}}{(1-q)\cdots (1-q^k)}\\
=&
 1+ \frac{q^2(1+q+\cdots+ q^{M-2}+tq^{M-1}+q^M)}{1-tq^{M+1}}\\
 &+
\frac{1-q^{M+1}}{1-tq^{M+1}}
\sum_{k=2}^M \frac{q^{k(k+1)}}{(1-q)\cdots (1-q^k)}.
\end{aligned}
$$
\end{prop}

\begin{proof}Divide both sides by $1-q^{M+1}$ and compute the coefficient of 
$t^N$ in the first two terms of both sides.
$$
\begin{aligned}
LHS=&q^{(M+1)N} \left(1+\frac{q^2}{1-q}\right)=
q^{(M+1)N} \left(\frac{1}{1-q}-q\right)\\
RHS=&\frac{1}{1-q^{M+1}}\left(q^2(1+q+\cdots+q^{M-2}+q^M)q^{(M+1)N}
+q^{(M+1)N}\right)\\
&=q^{(M+1)N}
\frac{(1+q^2+\cdots+q^M+q^{M+2})}{1-q^{M+1}}\\
&=q^{(M+1)N} \left(\frac{1}{1-q}-q\right)=LHS.
\end{aligned}
$$ 
\end{proof}

We can therefore give the part $M+1$ weight $t$, multiply the second
Rogers-Ramanujan identity by $\frac{1-q^{M+1}}{1-tq^{M+1}}$, and obtain a weighted 
generalization. 

\begin{theorem} 
\label{partM}
Suppose that $M+1$ is any part congruent to 2 or 3 modulo 5. Then 
$$
\begin{aligned}
1+&q^2\frac{1+q+\cdots+q^{M-2}+tq^{M-1}+q^M}{1-tq^{M+1}}+
\sum_{k=2}^M 
\frac{q^{k(k+1)}}{(1-q^2)\cdots(1-q^k)}\frac{[M+1]_q}{1-tq^{M+1}} \\
+&\sum_{k=M+1}^\infty 
\frac{q^{k(k+1)}}{(1-q)\cdots (1-q^M)(1-tq^{M+1})(1-q^{M+2})\cdots(1-q^k)} \\
&=
\frac{1-q^{M+1}}{1-tq^{M+1}}
\frac{1}{(1-q^2)(1-q^3)(1-q^7)(1-q^8)(1-q^{12})(1-q^{13})\cdots}.
\end{aligned}
$$
\end{theorem}

The combinatorial version of Theorem~\ref{partM} refines MacMahon's result by 
counting the number of $M+1$'s, or equivalently, finding the coefficient of $t^kq^N$ on 
both sides of Theorem~\ref{partM}. Note that the denominator factor $1-q,$ which 
accounts for the $1$'s in  $col^*(\lambda),$ has been replaced by
$$
\frac{[M+1]_q}{1-tq^{M+1}}.
$$

\begin{theorem} 
\label{generalminithm}
Let $M+1$ be a part congruent to $2$ or $3$ modulo 5. 
The number of partitions of $N\ge 1$ into parts congruent to $2$ or $3$ modulo $5$ with 
exactly $k$ $M+1$'s
is equal to the number of partitions $\lambda\in {\sf{Diff_2}}^*$ of $N$
such that either 
\begin{enumerate}
\item $\lambda$ consists of exactly one part of size $N=(M+1)k$ or \hfill\newline
 $N=(M+1)k+~i,$
$2\le i\le  M$ or $i=M+2,$
\item $\lambda$ has between $2$ and  $M$ parts, $col^*(\lambda)$ has between $(M+1)k$ 
and $(M+1)k+M$ $1$'s,
\item $\lambda$ has at least $M+1$ parts, and $col^*(\lambda)$ has $k$ $M+1$'s. 
\end{enumerate}
\end{theorem}

\begin{proof} The three cases correspond to the three terms on the sum side 
of Theorem~\ref{partM}. For example, in the second case, the second term 
has coefficient of $t^k$ of $q^{(M+1)k}[M+1]_q,$ which is interpreted as the 
contribution of the $1$'s of the missing denominator factor $1-q$.   
\end{proof}

\begin{exam} If $N=22$, $M+1=3$, and $k=2$ in Theorem~\ref{generalminithm}, 
the two equinumerous sets of partitions are
$$
\begin{array}{l|l|l}
\underline{\mu\in 2,3\  mod \ 5 }&\underline{\lambda\in {\sf{Diff_2}}^*}& \underline{col^*(\lambda)}\\
(12,3,3,2,2)&(16,6)& (2^4,1^{8})\\
(8,8,3,3)& (15,7)& (2^5,1^6)\\
(8,3,3,2,2,2,2)&(12,6,4)& (3^2,1^4)\\
(7,7,3,3,2)&(11,7,4)& (3^2,2,1^2)\\
(3,3,2,2,2,2,2,2,2,2)&(10,8,4)& (3^2,2^2).\\
\end{array}
$$
\end{exam}

Versions of Theorem~\ref{partM} and Theorem~\ref{generalminithm} for 
the first Rogers-Ramanujan identity may be given using
\begin{equation}
\begin{aligned}
\frac{1-q^{M+1}}{1-tq^{M+1}}& \sum_{k=0}^M \frac{q^{k^2}}{(1-q)\cdots (1-q^k)}\\
=& 1+q^1 \frac{(1+q+\cdots+ q^{M-2}+q^{M-1}+tq^M)}{1-tq^{M+1}}\\
 &+
\frac{1-q^{M+1}}{1-tq^{M+1}}
\sum_{k=2}^M \frac{q^{k^2}}{(1-q)\cdots (1-q^k)}.
\end{aligned}
\notag
\end{equation}

The weighted version is 
\begin{equation}
\label{partMeq}
\begin{aligned}
& 1+q^1 \frac{(1+q+\cdots+ q^{M-2}+q^{M-1}+tq^M)}{1-tq^{M+1}}
+
\sum_{k=2}^M q^{k^2}\frac{[M+1]_q}{1-tq^{M+1}}
\frac{1}{(1-q^2)\cdots (1-q^k)}\\
&+\sum_{k=M+1}^\infty \frac{q^{k^2}}
{(1-q)\cdots (1-q^M)(1-tq^{M+1})(1-q^{M+2})\cdots(1-q^k)}\\
&=
\frac{1-q^{M+1}}{1-tq^{M+1}}
\frac{1}{(1-q^1)(1-q^4)(1-q^6)(1-q^9)(1-q^{11})(1-q^{14})\cdots}.
\end{aligned}
\end{equation}

This gives the following refinement of MacMahon's result.

\begin{theorem} 
\label{generalmini14thm}
Let $M+1\ge2$ be a part congruent to $1$ or $4$ modulo 5. 
The number of partitions of $N\ge 1$ into parts congruent to $1$ or $4$ modulo $5$ with 
exactly $k$ $M+1$'s
is equal to the number of partitions $\lambda\in {\sf{Diff_2}}$ of $N$
such that either 
\begin{enumerate}
\item $\lambda$ consists of exactly one part of size  $N=(M+1)k+i,$
$0\le i\le  M,$ or
\item $\lambda$ has between $2$ and  $M$ parts, $col(\lambda)$ has between $(M+1)k$ 
and $(M+1)k+M$ $1$'s,
\item $\lambda$ has at least $M+1$ parts, and $col(\lambda)$ has $k$ $M+1$'s. 
\end{enumerate}
\end{theorem}

\begin{exam} If $N=23$, $M+1=4$, and $k=3$ in Theorem~\ref{generalmini14thm},
the two equinumerous sets of partitions are
$$
\begin{array}{l|l|l}
\underline{\mu\in 1,4\  mod \ 5 }&\underline{\lambda\in {\sf{Diff_2}}}& \underline{col(\lambda)}\\
(11,4,4,4)&(20,3)& (2^2,1^{15})\\
(9,4,4,4,1,1)& (19,4)& (2^3,1^{13})\\
(6,4,4,4,1,1,1,1,1)&(19,3,1)& (1^{14})\\
(4,4,4,1^{11})&(18,4,1)& (2,1^{12}).\\
\end{array}
$$
\end{exam}

\section{Refinements with two parts}

One may ask for a  version of the Rogers-Ramanujan identities which weights
more than one part. In this section we give two types of results. The first type has two 
weighted parts: $2$ and $M\ge 7$ for $2$ or $3$ modulo $5$ (Theorem~\ref{twopartM}), or parts 
$1$ and $M\ge 6$ even for $1$ or $4$ modulo $5$ (Theorem~\ref{twopart14}). 
The second type allows either $2$'s or $3$'s. 

What is key here is the positivity of the numerator polynomials 
in the modified Rogers-Ramanujan expansions. 

First we give an analogue of \eqref{weirdeq} for parts $2$ and $M$.
\begin{prop} 
\label{parts2Meq}
For any positive integer $M$,
$$
\begin{aligned}
&\frac{(1-q^2)(1-q^{M})}{(1-tq^2)(1-wq^{M})}
\sum_{k=0}^2 \frac{q^{k(k+1)}}{(1-q)\cdots(1-q^k)}\\
=& 1+q^2\frac{t+q}{1-tq^2}+\frac{q^6}{(1-tq^2)(1-wq^M)}
\left( [M]_q+(w-1)(q^{M-3}+q^{M-6})\right).
\end{aligned}
$$
\end{prop}

Note that if $M\ge 6,$ the numerator polynomial for $q^6$ 
has positive coefficients, and 
$$
\frac{( [M]_q+(w-1)(q^{M-3}+q^{M-6}))}{1-wq^M}=\frac{1}{1-q} \quad {\text{if }} w=1.
$$
So we may use this quotient to represent the $1$'s in $col^*(\lambda)$
when the difference partition has exactly two parts.

Theorem~\ref{twopartM} uses Proposition~\ref{parts2Meq} and gives an 
alternative sum expression (when $t=1$) to Theorem~\ref{partM} for 
the single weighted part $M$.

\begin{theorem} 
\label{twopartM}
Suppose that $M\ge 7$ is any part congruent to 2 or 3 modulo 5. Then 
$$
\begin{aligned}
1+&q^2\frac{t+q}{1-tq^{2}}+
\frac{q^6}{(1-tq^2)(1-wq^M)}
\left( [M]_q+(w-1)(q^{M-3}+q^{M-6})\right) \\
+&\sum_{k=3}^{M-1} 
\frac{q^{k(k+1)}}{(1-tq^2)(1-q^3)\cdots (1-q^{k})}\frac{[M]_q}{1-wq^M} \\
+&\sum_{k=M}^\infty 
\frac{q^{k(k+1)}}{(1-q)(1-tq^2)\cdots (1-q^{M-1})(1-wq^{M})(1-q^{M+1})\cdots(1-q^k)} \\
&=
\frac{1-q^{M}}{1-wq^{M}}
\frac{1}{(1-tq^2)(1-q^3)(1-q^7)(1-q^8)(1-q^{12})(1-q^{13})\cdots}.
\end{aligned}
$$
\end{theorem}
  
\begin{cor} 
\label{general2partcor}
Let $M\ge 7$ be a part congruent to $2$ or $3$ modulo 5. 
The number of partitions of $N\ge 1$ into parts congruent to $2$ or $3$ modulo $5$ with 
exactly $k$ $M$'s and $j$ $2$'s
is equal to the number of partitions $\lambda\in {\sf{Diff_2}}^*$ of $N$
such that either 
\begin{enumerate}
\item $\lambda$ consists of exactly one part of size $N=2j$ or 
 $N=2j+3,$ and $k=0,$
\item $\lambda$ has exactly two parts, $col^*(\lambda)$ has exactly $j$ $2$'s, 
and either $Mk+i$ $1$'s, $0\le i\le M-1$, $i\neq M-3,M-6$, $k\ge 0,$ or $M(k-1)+M-3$, or 
$M(k-1)+M-6$ $1$'s, $k\ge 1$,  
\item $\lambda$ has between three and  $M-1$ parts, $col^*(\lambda)$ has $j$ $2$'s and between
$Mk$ and $Mk+M-1$ $1$'s,
\item $\lambda$ has at least $M$ parts, and $col^*(\lambda)$ has $j$ $2$'s and $k$ $M$'s. 
\end{enumerate}
\end{cor}

A version of Theorem~\ref{twopartM} and Corollary~\ref{general2partcor} for the 
first Rogers-Ramanujan identity using parts $1$ and an even $M\ge 4$ may be given using
$$
\begin{aligned}
&\frac{(1-q^1)(1-q^{M})}{(1-tq^1)(1-wq^{M})}
\sum_{k=0}^2 \frac{q^{k^2}}{(1-q)\cdots(1-q^k)}\\
=& 1+q^1\frac{t}{1-tq}+\frac{q^4}{(1-tq)(1-wq^M)}
\left( [M/2]_{q^2}+(w-1)q^{M-4}\right).
\end{aligned}
$$
In this case
$$
\frac{([M/2]_{q^2}+(w-1)q^{M-4})}{1-wq^M}= \frac{1}{1-q^2} {\text{ if }} w=1.
$$
This time we use the quotient to represent $2$'s in $col(\lambda)$ when 
$\lambda$ has exactly two parts.

\begin{prop} 
\label{twopart14}
If $M\ge 4$ is even and congruent to 1 or 4 modulo $5$, then
$$
\begin{aligned}
1+&q\frac{t}{1-tq}+
\frac{q^4}{1-tq}
\frac{\left( [M/2]_{q^2}+(w-1)q^{M-4}\right)}{1-wq^M} \\
+&\sum_{k=3}^{M-1} 
\frac{q^{k^2}}{(1-tq)(1-q^3)\cdots (1-q^{k})}\frac{[M/2]_{q^2}}{1-wq^M} \\
+&\sum_{k=M}^\infty 
\frac{q^{k^2}}{(1-tq)(1-q^2)\cdots (1-q^{M-1})(1-wq^{M})(1-q^{M+1})\cdots(1-q^k)} \\
&=
\frac{1-q^{M}}{1-wq^{M}}
\frac{1}{(1-tq)(1-q^4)(1-q^6)(1-q^9)(1-q^{12})(1-q^{13})\cdots}.
\end{aligned}
$$
\end{prop}

\begin{cor} 
\label{general2part14cor}
Let $M\ge 4$ be a even part congruent to $1$ or $4$ modulo 5. 
The number of partitions of $N\ge 1$ into parts congruent to $1$ or $4$ modulo $5$ with 
exactly $k$ $M$'s and $j$ $1$'s
is equal to the number of partitions $\lambda\in {\sf{Diff_2}}$ of $N$
such that either 
\begin{enumerate}
\item $\lambda$ consists of exactly one part of size  $N=j,$ and $k=0,$
\item $\lambda$ has exactly two parts, $col(\lambda)=2^{\ell}1^j,$ 
where $\ell=Mk/2+i,$ \newline $0\le i\le M/2-1$, $i\neq M/2-2,$ $k\ge 0,$ or $\ell=Mk/2-2,$
$k\ge 1$,  
\item $\lambda$ has between three and  $M-1$ parts, $col(\lambda)$ has $j$ $1$'s and between
$Mk/2$ and $Mk/2+M/2-1$ $2$'s,
\item $\lambda$ has at least $M$ parts, and $col(\lambda)$ has $j$ $1$'s and $k$ $M$'s. 
\end{enumerate}
\end{cor}

Theorem~\ref{twopartM} does not apply to the part sizes $2$ and $3$.  
We have two such results. 
We do not state the preliminary rational function identity analogous to \eqref{weirdeq}, 
and give only the final results.

\begin{prop} 
\label{firsttw}
A $t,w$-refinement of the second Rogers-Ramanujan identity for part sizes $2$ and $3$ is
$$
\begin{aligned}
1+&q^2\frac{(t+wq)}{1-tq^2}+
q^6\frac{(w^2+q+q^2)}{(1-tq^2)(1-wq^3)}\\
&+\sum_{m=3}^\infty 
\frac{q^{m(m+1)}}{(1-q)(1-tq^2)(1-wq^3)(1-q^4)\cdots (1-q^m)} \\
&=\frac{1}{(1-tq^2)(1-wq^3)(1-q^7)(1-q^8)(1-q^{12})(1-q^{13})\cdots}.
\end{aligned}
$$
\end{prop}

These positive expansions are not unique.
\begin{prop} 
\label{secondtw}
Another $t,w$-refinement of the second Rogers-Ramanujan identity 
for part sizes $2$ and $3$ is
$$
\begin{aligned}
1+&q^2\frac{(t+wq +t^2q^2)}{1-wq^3}+
q^6\frac{(q + q^2 + t^3)}{(1-tq^2)(1-wq^3)}\\
&+\sum_{m=3}^\infty 
\frac{q^{m(m+1)}}{(1-q)(1-tq^2)(1-wq^3)(1-q^4)\cdots (1-q^m)} \\
&=\frac{1}{(1-tq^2)(1-wq^3)(1-q^7)(1-q^8)(1-q^{12})(1-q^{13})\cdots}.
\end{aligned}
$$
\end{prop}

Because Propositions ~\ref{firsttw} and \ref{secondtw} have positive coefficients in the 
numerators of the sum side, distinct combinatorial versions could be given.

\section{Refinements with three or four parts}

In this section we give refinements for up to four small parts. The preliminary 
rational function identities were found experimentally and then verified by computer. 
We begin with parts $2$, $3$, and $7$.

\begin{theorem} 
\label{twvthm}
A $t,w,v$-refinement of the second Rogers-Ramanujan identity 
for part sizes $2,$ $3$ and $7$ is
$$
\begin{aligned}
1+&q^2\frac{(t+wq)}{1-tq^2}+
q^6\frac{(w^2+vq+q^2)}{(1-tq^2)(1-wq^3)}+
q^{12}\frac{(1+q+v^2q^2+vq^3+q^4+q^5+q^6)}{(1-tq^2)(1-wq^3)(1-vq^7)}\\
&+q^{20}\frac{(1+q+q^2+q^3+q^4+q^5+q^6)}{(1-tq^2)(1-wq^3)(1-q^4)(1-vq^7)}\\
&+q^{30}\frac{(1+q+q^2+q^3+q^4+q^5+q^6)}{(1-tq^2)(1-wq^3)(1-q^4)(1-q^5)(1-vq^7)}\\
&+q^{42}\frac{(1+q+q^2+q^3+q^4+q^5+q^6)}{(1-tq^2)(1-wq^3)(1-q^4)(1-q^5)(1-q^6)(1-vq^7)}\\
&+\sum_{m=7}^\infty 
\frac{q^{m(m+1)}}{(1-q)(1-tq^2)(1-wq^3)(1-q^4)(1-q^5)(1-q^6)(1-vq^7)\cdots (1-q^m)} \\
&=\frac{1}{(1-tq^2)(1-wq^3)(1-vq^7)(1-q^8)(1-q^{12})(1-q^{13})\cdots}.
\end{aligned}
$$
\end{theorem}

Theorem~\ref{twvthm} could be considered as an expansion to insert 
$1-tq^2$, $1-wq^3$, and $1-vq^7$ 
consecutively into the sum side of the second Rogers-Ramanujan identity,
replacing $1-q$, $1-q^2$, and $1-q^3.$ 

Because
$$
q^2\frac{t + wq}{1 - tq^2}+q^6
\frac{(w^2 + vq + q^2)}{(1 - tq^2)(1 - wq^3)}
= 
q^2\frac{(t+wq +t^2q^2)}{1 - wq^3} + 
  q^6\frac{(vq + q^2 + t^3)}{(1 - tq^2)(1 - wq^3)},
$$
Theorem~\ref{twvthm} can be rewritten in a form that inserts 
$1-wq^3$, $1-tq^2$,  and $1-vq^7$ consecutively. 

Also
$$
\begin{aligned}
&q^2\frac{(t+wq)}{1-tq^2}+
q^6\frac{(w^2+vq+q^2)}{(1-tq^2)(1-wq^3)}+
q^{12}\frac{(1+q+v^2q^2+vq^3+q^4+q^5+q^6)}{(1-tq^2)(1-wq^3)(1-vq^7)}=\\
&q^2\frac{(t+qw+q^2t^2+q^3tw+q^4t^3+q^5v+q^6)}{1-vq^7}\\
+&q^6\frac{(w^2+qt^2w+q^2t^4+w^3q^3+tq^4+wq^5+w^4q^6)}{(1-tq^2)(1-vq^7)}\\
+&q^{12}\frac{(1+q+q^2w^2+q^3w^5+q^4+q^5+q^6)}{(1-tq^2)(1-wq^3)(1-vq^7)}
\end{aligned}
$$
which gives an insertion ordering of $1-vq^7$, $1-tq^2$, and $1-wq^3$ consecutively.
Each of these three orderings for $t,w,v$ refinements have positive expansions and distinct combinatorial intepretations. 

We give the interpretation which corresponds to Theorem~\ref{twvthm}. 

\begin{cor} 
\label{firstbigcomb}
The number of partitions of $N\ge 1$ into parts congruent to $2$ or $3$ 
modulo 5 with $k$ 2's, $j$ 3's and $\ell$ 7's is equal to the number of partitions 
$\lambda\in {\sf{Diff_2}}^*$ of $N$ such that 
\begin{enumerate}
\item if $\lambda$ has at least seven parts, $col^*(\lambda)$ has $k$ 2's, $j$ 3's and $\ell$ 7's,
\item if $\lambda$ has four, five or six parts, 
$col^*(\lambda)$ has $k$ 2's, $j$ 3's and between $7\ell$ and $7\ell+6$ 1's, 
\item if $\lambda$ has three parts, $col^*(\lambda)$ has $k$ 2's, $j$ 3's and either 
$7\ell,$ $7\ell+1$, $7\ell-12$, $7\ell-4$, $7\ell+4$, $7\ell+5$, or $7\ell+6$ 1's,
\item if $\lambda$ has two parts,
\begin{enumerate}
\item $col^*(\lambda)=2^{k}1^{3(j-2)}$ or $2^k1^{3j+2}$ if $\ell=0,$  or 
\item $col^*(\lambda)=2^{k}1^{3j+1}$ if $\ell=1$,
\end{enumerate}
\item if $\lambda$ has one part, 
\begin{enumerate}
\item $col^*(\lambda)=1^{2k-2}$, if $k\ge 1$, and $j=\ell=0,$
\item $col^*(\lambda)=1^{2k+1}$, if $k\ge 0$, and $j=1$ and $\ell=0.$
\end{enumerate}
\end{enumerate}
\end{cor}

\begin{proof} We need to interpret the coefficient of $q^Nt^kw^jv^\ell$ in the 
sum side of Theorem~\ref{twvthm}.

If $\lambda$ has at least seven parts, the sum side of Theorem~\ref{twvthm} 
has denominator factors of $1-tq^2$, $1-wq^3$, $1-vq^7$ which directly become parts in $col^*(\lambda)$.

If $\lambda$ has between four and six parts, the factor $1/(1-q)$ in the sum side, 
which represents the $1$'s 
in $col^*(\lambda)$, has been replaced by 
$$
\frac{(1+q+q^2+q^3+q^4+q^5+q^6)}{1-vq^7}.
$$ 
So we have the term $v^\ell$ when 
$col^*(\lambda)$ has between $7\ell$ and $7\ell+6$ 1's.

If $\lambda$ has three parts, the factor $1/(1-q)$ in the sum side, 
which represents the $1$'s, has been replaced by 
$$
\frac{(1+q+v^2q^2+vq^3+q^4+q^5+q^6)}{1-vq^7}.
$$ 
So we have the 
term $v^\ell$ when $col^*(\lambda)$ has  $7\ell,$ $7\ell+1$, 
$7\ell-12$, $7\ell-4$, $7\ell+4$, $7\ell+5$, or $7\ell+6$ 1's.

The cases when $\lambda$ has one or two parts are done similarly.
\end{proof}

\begin{exam} If $N=22$ Corollary~\ref{firstbigcomb} induces a bijection, which is given below.
$$
\begin{array}{c|l|l|l}
\underline{\mu\in 2,3\  mod \ 5 }&\underline{\lambda\in {\sf{Diff_2}}^*}& \underline{col^*(\lambda)}&
\underline{(k,j,\ell)}\\
(22)&(10,6,4,2)&(4,1,1)&(0,0,0)\\
(18,2^2)&(14,6,2)&(2^2,1^6)&(2,0,0)\\
(17,3,2)&(13,6,3)&(3,2,1^5)&(1,1,0)\\
(13,7,2)&(15,5,2)&(2,1^8)&(1,0,1)\\
(13,3^3)&(10,7,5)&(3^3,1)&(0,3,0)\\
(13,3,2^3)&(11,8,3)&(3,2^3,1)&(3,1,0)\\
(12,8,2)&(9,7,4,2)&(4,2)&(1,0,0)\\
(12,7,3)&(14,5,3)&(3,1^7)&(0,1,1)\\
(12,3^2,2^2)&(10,8,4)&(3^2,2^2)&(2,2,0)\\
(12,2^5)&(11,9,2)&(2^5)&(5,0,0)\\
(8^2,3^2)&(12,6,4)&(3^2,1^4)&(0,2,0)\\
(8^2,2^3)&(13,7,2)&(2^3,1^4)&(3,0,0)\\
(8,7^2)&(16,4,2)&(1^{10})&(0,0,2)\\
(8,7,3,2^2)&(12,7,3)&(3,2^2,1^3)&(2,1,1)\\
(8,3^4,2)&(19,3)&(2,1^{14})&(1,4,0)\\
(8,2^7)&(13,9)&(2^7,1^2)&(7,0,0)\\
(7^2,3^2,2)&(11,7,4)&(3^2,2,1^2)&(1,2,2)\\
(7^2,2^4)&(12,8,2)&(2^4,1^2)&(4,0,2)\\
(7,3^5)&(20,2)&(1^{16})&(0,5,1)\\
(7,3^3,2^3)&(17,5)&(2^3,1^{10})&(3,3,1)\\
(7,3,2^6)&(14,8)&(2^6,1^4)&(6,1,1)\\
(3^6,2^2)&(18,4)&(2^2,1^{12})&(2,6,0)\\
(3^4,2^5)&(15,7)&(2^5,1^6)&(5,4,0)\\
(3^2,2^8)&(12,10)&(2^8)&(8,2,0)\\
(2^{11})&(22)&(1^{20})&(11,0,0)\\
\end{array}
$$
\end{exam}

The most general result of this type that we found has four independent weights,
for the $2$'s, $3$'s, $7$'s, and $8$'s. 

\begin{theorem} 
\label{twvx23theorem}
A $t,w,v,x$-refinement of the second Rogers-Ramanujan identity for 
part sizes $2,$ $3,$ $7$ and $8$ is
$$
\begin{aligned}
1&+q^2\frac{(t+wq)}{1-tq^2}+
q^6\frac{(w^2+vq+xq^2)}{(1-tq^2)(1-wq^3)}+
q^{12}\frac{(1+q+v^2q^2+xvq^3+x^2q^4+q^5+q^6)}{(1-tq^2)(1-wq^3)(1-vq^7)}\\
&+q^{20}\frac{(x+xq+q^2+q^3+(1+x^3)q^4+(1+x)q^5+(1+x)q^6+q^7+q^8+q^9+q^{10})}
{(1-tq^2)(1-wq^3)(1-xq^8)(1-vq^7)}\\
&+q^{30}\frac{(1+q+q^2+q^3+q^4+q^5+q^6)(1+q^4)}
{(1-tq^2)(1-wq^3)(1-xq^8)(1-q^5)(1-vq^7)}\\
&+q^{42}\frac{(1+q+q^2+q^3+q^4+q^5+q^6)(1+q^4)}
{(1-tq^2)(1-wq^3)(1-xq^8)(1-q^5)(1-q^6)(1-vq^7)}\\
&+q^{56}\frac{(1+q^4)}
{(1-q)(1-tq^2)(1-wq^3)(1-xq^8)(1-q^5)(1-q^6)(1-vq^7)}+\\
&\sum_{m=8}^\infty 
\frac{q^{m(m+1)}}{(1-q)(1-tq^2)(1-wq^3)(1-q^4)(1-q^5)(1-q^6)(1-vq^7)
(1-xq^8)\cdots (1-q^m)} \\
&=\frac{1}{(1-tq^2)(1-wq^3)(1-vq^7)(1-xq^8)(1-q^{12})(1-q^{13})\cdots}.
\end{aligned}
$$
\end{theorem}

An analogous result with 4 parameters for the first Rogers-Ramanujan
identity is
\begin{theorem} 
\label{twvx14thm}
A $t,w,v,x$-refinement of the first Rogers-Ramanujan
identity for part sizes $1,$ $4,$ $6$ and $9$ is
$$
\begin{aligned}
1&+q\frac{t}{1-tq}+
q^4\frac{(w+vq^2)}{(1-tq)(1-wq^4)}+
q^{9}\frac{(x+q^2+v^2q^3+q^5)}{(1-tq)(1-wq^4)(1-vq^6)}\\
&+q^{16}\frac{(1+x^2q^2+q^3+xq^4+q^5+q^6 +xq^7+q^8+q^{10})}
{(1-tq)(1-wq^4)(1-vq^6)(1-xq^9)}\\
&+q^{25}\frac{(1+q^2 +q^3 +q^4 +q^5 +q^6 +q^7+q^8 +q^{10})}
{(1-tq)(1-wq^4)(1-vq^6)(1-xq^9)(1-q^5)}\\
&+q^{36}\frac{(1+q^2 +q^3 +q^4 +q^5 +q^6 +q^7+q^8 +q^{10})}
{(1-tq)(1-wq^4)(1-vq^6)(1-xq^9)(1-q^5)(1-q^6)}\\
&+q^{49}\frac{(1+q^2 +q^3 +q^4 +q^5 +q^6 +q^7+q^8 +q^{10})}
{(1-tq)(1-wq^4)(1-vq^6)(1-xq^9)(1-q^5)(1-q^6)(1-q^7)}\\
&+q^{64}\frac{(1+q^2 +q^3 +q^4 +q^5 +q^6 +q^7+q^8 +q^{10})}
{(1-tq)(1-wq^4)(1-vq^6)(1-xq^9)(1-q^5)(1-q^6)(1-q^7)(1-q^8)}\\
+&\sum_{m=9}^\infty 
\frac{q^{m^2}}{(1-tq^1)(1-q^2)(1-q^3)(1-wq^4)(1-q^5)(1-vq^6)(1-q^7)
(1-q^8)(1-xq^9)\cdots (1-q^m)} \\
&=\frac{1}{(1-tq^1)(1-wq^4)(1-vq^6)(1-xq^9)(1-q^{11})(1-q^{14})\cdots}.
\end{aligned}
$$
\end{theorem}

A combinatorial interpretation of Theorem~\ref{twvx14thm} may be given, but 
for simplicity we provide one for the case $x=1.$
\begin{cor} 
\label{bigcomb}
The number of partitions of $N\ge 1$ into parts congruent to $1$ or $4$ 
modulo 5 
with $k$ 1's, $j$ 4's and $\ell$ 6's is equal to the number of partitions 
$\lambda\in {\sf{Diff_2}}$ of $N$ such that 
\begin{enumerate}
\item if $\lambda$ has at least six parts, $col(\lambda)$ has $k$ 1's, $j$ 4's and $\ell$ 6's,
\item if $\lambda$ has four or five parts, 
$col(\lambda)$ has $k$ 1's, $j$ 4's and $2\ell$ or $2\ell+1$ 3's, 
\item if $\lambda$ has three parts,
\begin{enumerate}
\item $col(\lambda)=3^{2\ell}2^{2j}1^k,$  or 
\item $col(\lambda)=3^{2\ell}2^{2j+1}1^k,$ or
\item $col(\lambda)=3^{2\ell-3}2^{2j}1^k,$ $\ell\ge 2,$ or 
\item $col(\lambda)=3^{2\ell+1}2^{2j+1}1^k.$  
\end{enumerate}
\item if $\lambda$ has two parts,
\begin{enumerate}
\item $col(\lambda)=2^{2j-2}1^k,$  $j\ge 1,$ and $\ell=0,$  or 
\item $col(\lambda)=2^{2j+1}1^k$  and $\ell=1$,
\end{enumerate}
\item if $\lambda$ has one part, $col(\lambda)=1^{k-1},$  $k\ge 1$, and $j=\ell=0.$
\end{enumerate}
\end{cor}

\begin{proof} We need to interpret the coefficient of $q^Nt^kw^jv^\ell$ in the 
sum side of Theorem~\ref{twvx14thm}.

First note that the choice of $x=1$ allows the final sum in 
Theorem~\ref{twvx14thm} to start at $m=6$ rather than $m=9,$
because
$$
\frac{(1+q^2+q^3+q^4+q^5+q^6+q^7+q^8+q^{10})}{1-q^9}=
\frac{(1-q^6)}{(1-q^2)(1-q^3)}.
$$ 
This is the first case.

The other cases are established as in the proof of Corollary~\ref{firstbigcomb}. 
We will explain case 3, when $\lambda$ has three parts, whose corresponding term in 
Theorem~\ref{twvx14thm} is
$$
q^9 \frac{(1+q^2+v^2q^3+q^5)}{(1-tq)(1-wq^4)(1-vq^6)}.
$$
We need the interpretation for the coefficient of $q^Nt^kw^jv^\ell$ 
in the Taylor series of this rational function. The power $q^9$ 
represents removing $1+3+5$ from $\lambda$ which has three parts, and 
the remaining quotient is the generating function for 
$col(\lambda).$ It is a weighted version of 
$1/(1-q)(1-q^2)(1-q^3).$ 

The four terms in the numerator correspond to the four subcases of case 3. The 
first subcase, ``1", has 
an arbitrary number of $1$'s, $4$'s, and 6's, which may be changed 
to 1's and an even number of $2$'s and $3$'s in $col(\lambda)$. The second subcase
``$q^2$", corresponds to an odd number of $2's$. The third subcase
``$v^2q^3$" has a term $v^\ell q^{6\ell-9}$, $\ell\ge 2,$ which corresponds to 
$2\ell-3$ $3$'s. Finally the term ``$q^5$" allows an odd number of $2$'s and an odd 
number of $3$'s.
\end{proof}

\begin{exam} If $N=19$ Corollary~\ref{bigcomb} induces a bijection, which is given below.
$$
\begin{array}{l|l|l|l}
\underline{\mu\in 1,4\  mod \ 5 }&\underline{\lambda\in {\sf{Diff_2}}}& \underline{col(\lambda)}&
\underline{(k,j,\ell)}\\
(19)&(8,6,4,1)&(3)&(0,0,0)\\
(16,1^3)&(10,5,3,1)&(1^3)&(3,0,0)\\
(14,4,1)&(10,7,2)&(3,2^3,1)&(1,1,0)\\
(14,1^5)&(12,5,2)&(3,2,1^5)&(5,0,0)\\
(11,6,1^2)&(10,6,3)&(3^2,2,1^2)&(2,0,1)\\
(11,4^2)&(10,8,1)&(2^5)&(0,2,0)\\
(11,4,1^4)&(12,6,1)&(2^3,1^4)&(4,1,0)\\
(11,1^8)&(14,4,1)&(2,1^8)&(8,0,0)\\
(9^2,1)&(9,6,3,1)&(2,1)&(1,0,0)\\
(9,6,4)&(9,7,3)&(3^2,2^2)&(0,1,1)\\
(9,6,1^4)&(11,5,3)&(3^2,1^4)&(4,0,1)\\
(9,4^2,1^2)&(11,7,1)&(2^4,1^2)&(2,2,0)\\
(9,4,1^6)&(13,5,1)&(2^2,1^6)&(6,1,0)\\
(9,1^{10})&(15,3,1)&(1^{10})&(10,0,0)\\
(6^3,1)&(9,6,4)&(3^3,1)&(1,0,3)\\
(6^2,4,1^3)&(11,6,2)&(3,2^2,1^3)&(3,1,2)\\
(6^2,1^7)&(13,4,2)&(3,1^{7})&(7,0,2)\\
(6,4^3,1)&(11,8)&(2^7,1)&(1,3,1)\\
(6,4^2,1^5)&(13,6)&(2^5,1^{5})&(5,2,1)\\
(6,4,1^9)&(15,4)&(2^3,1^9)&(9,1,1)\\
(6,1^{13})&(17,2)&(2,1^{13})&(13,0,1)\\
(4^4,1^3)&(12,7)&(2^6,1^3)&(3,4,0)\\
(4^3,1^7)&(14,5)&(2^4,1^{7})&(7,3,0)\\
(4^2,1^{11})&(16,3)&(2^2,1^{11})&(11,2,0)\\
(4,1^{15})&(18,1)&(1^{15})&(15,1,0)\\
(1^{19})&(19)&(1^{18})&(19,0,0)\\
\end{array}
$$
\end{exam}

\section{Three specializations}

By specializing the weights $t$, $w$, $v$, and $x$ one may restrict or 
change the part sizes. We give three such examples in this section.

First let's take $t=v=1$ and $w=x=0$ in Theorem~\ref{twvx23theorem}, 
so no $3$'s and $8$'s are allowed. 

\begin{theorem} 
 \label{spec1}
 The number of partitions of $N\ge 1$ into parts congruent to $2$ or $3$ 
modulo $5$ with no $3$'s and no $8$'s is equal to the number of partitions 
$\lambda\in {\sf{Diff_2}}^*$ of $N$ such that
\begin{enumerate}
\item $\lambda$ has $m\ge 8$ parts, $col^*(\lambda)$ has no $3$'s and no $8$'s, 
\item  $\lambda$ has $m$ parts, $5\le m\le 7,$  $col^*(\lambda)$ has at most one $4$ 
and no $3$'s, 
\item  $\lambda$ has four parts,  and $col^*(\lambda)=3^j2^k1^\ell$ has no $4$'s,  
$j\le 2$, $\ell\equiv 2,3,4 \mod 7$,
\item  $\lambda$ has three parts,  $col^*(\lambda)=2^k1^\ell$ has no $3$'s, and $\ell\not\equiv  3,4 \mod 7,$
\item  $\lambda$ has two parts, and $col^*(\lambda)$ has exactly one $1$, 
\item  $\lambda$ has one even part.
\end{enumerate}  
\end{theorem}

\begin{proof} These cases are the terms on the sum side of 
Theorem~\ref{twvx23theorem} in reverse order. Let $m$ be the 
number of parts of $\lambda.$

If  $m\ge 8,$ the choice of $w=x=0$ eliminates parts $3$ and $8$ in $col^*(\lambda).$

If $5\le m\le 7$, the $m$th term can be rewritten as
$$
q^{m(m+1)} \frac{(1+q^4)}{(1-q)(1-q^2)(1-q^5)\cdots (1-q^m)},
$$
so $col^*(\lambda)$ has no $3$'s and at most one $4$.

If $m=4,$ the $m$th term can be rewritten as
$$
q^{20}\frac{(q^2+q^3+q^4)}{1-q^7}\frac{1}{1-q^2}(1+q^3+q^6)
$$
so $col^*(\lambda)=3^j2^k1^\ell $ has no $4$'s, at most two $3$'s, and $\ell\equiv 2,3,4 \mod 7.$

If $m=3,$ the $m$th term can be rewritten as 
$$
q^{12}\frac{(1+q+q^2+q^5+q^6)}{1-q^7}\frac{1}{1-q^2}
$$
which gives the stated choices. The  cases $m=2$ and $m=1$ are clear.
\end{proof}

For the second choice, take $t=w=v=x=0$ in Theorem~\ref{twvx14thm} to knock out 
the part sizes $1$,  $4$, $6$, and $9$. The first five terms on the sum 
side are polynomials of degree at most $26,$ so do not contribute if $N\ge 27.$ We use
$$
1+q^2+q^3+q^4+q^5+q^6+q^7+q^8+q^{10}=(1+q^2+q^4)(1+q^3+q^6).
$$

\begin{theorem} 
\label{spec2}
The number of partitions of $N\ge 27$ into parts congruent to $1$ or $4$ modulo 5, 
whose smallest part is at least $11$ is equal to the number of partitions 
$\lambda\in {\sf{Diff_2}}$ of $N$ such that
\begin{enumerate}
\item $\lambda$ has $m$ parts, $5\le m\le 8$, $col(\lambda)$ has at most 
two 2's, at most two 3's, no 1's and no 4's, or
\item  $\lambda$ has $m$ parts, $m\ge 9,$ $col(\lambda)$ has no $1$'s, $4$'s, $6$'s, or $9$'s.
\end{enumerate}  
\end{theorem}

Finally let's specialize Proposition~\ref{firsttw} to $t=1$ and $w=q^2$. 
Now the parts which were $2$ and $3$ 
modulo $5$ have no $3$'s but $5$'s are allowed. 

\begin{theorem} 
\label{spec3}
The number of partitions of $N$ into parts congruent to $2$ or $3$ modulo 5 or 
$5$'s but no $3$'s 
is equal to the number of partitions $\lambda\in {\sf{Diff_2}}^*$ of $N$ such that
\begin{enumerate}
\item $\lambda$ has one part, $\lambda\neq 3,$  or
\item  $\lambda$ has two parts, and the number of $1$'s in $col^*(\lambda)$ is 
congruent to $1$, $2$, or $4$ modulo 5, or
\item  $\lambda$ has at least three parts, $col^*(\lambda)$ has at least as many $2$'s as $3$'s.
\end{enumerate}  
\end{theorem}

\begin{proof} This specialization gives for the sum side of Proposition~\ref{secondtw} is
$$
\begin{aligned}
1+&q^2\left( \frac{1}{1-q}-q\right)+\frac{q^6}{1-q^2} \frac{q+q^2+q^4}{1-q^5}\\
+&\sum_{m=3}^\infty \frac{q^{m(m+1)}}{(1-q)(1-q^2)(1-q^5)(1-q^4)\cdots (1-q^m)}.
\end{aligned}
$$
If $\lambda$ has at least three parts, 
we interpret the denominator factor $(1-q^5)$ as melded parts $32.$ Then the number of 
$2$'s in $col^*(\lambda)$ must be at least as great as the number of 
$3$'s in $col^*(\lambda)$. 
\end{proof}

The specialization $t=1$, $w=q^2,$ in Proposition~\ref{secondtw} 
gives a version of Theorem~\ref{spec3} with a 
different subset of ${\sf{Diff_2}}^*.$

\section{Remarks}

Once the preliminary rational function identities are 
experimentally discovered, they are easily proven
using computer algebra. Upon 
multiplying the  Rogers-Ramanujan identities by the appropriate rational function,  
a weighted version is obtained. 
We do not have an algorithm which generates 
such identities, nor do we have a priori explanations for the crucial positivity of the 
numerator coefficients. The non-uniqueness of the identities 
adds to their difficulty. 
We do not know if a particular specialization of the weights
simplifies the numerator polynomials. 

We were led to consider such identities while 
considering a possible direct Rogers-Ramanujan bijection. 
That bijection remains elusive, see \cite{GM}, \cite{BP}.
Our Theorems indicate that parts on the modulo 5 side mostly become 
columns on the difference two side, with some number theoretic initial
cases.

\end{document}